\documentclass[11pt,a4paper]{article}
\usepackage[margin=3cm]{geometry}
\usepackage{graphicx} 
\usepackage{tikz}
\usepackage{amsmath}
\usepackage{float}
\usepackage{amsfonts}
\usepackage{amssymb}
\usepackage{xcolor}
\usepackage{amsthm}
\usepackage{needspace}
\usepackage{microtype}
\usepackage{hyperref}
\hypersetup{pdftitle={Spectral comparison on Specht modules and mixing of the interchange process},pdfauthor={Gil Alon}}
\allowdisplaybreaks[1]

\newcommand{\norm}[1]{\lVert #1\rVert}
\newcommand{\one}{\mathbf1}
\newcommand{\Om}{\Omega}
\newcommand{\strep}{\mathrm{st}}
\newcommand{\SG}{\mathcal S_G}
\newcommand{\CG}{\mathcal C_G}
\newcommand{\Eswaps}{\mathcal E_G^{\mathrm{swaps}}}

\def\R{\mathbb{R}}

\def\C{\mathbb{C}}

\theoremstyle{plain}
\newtheorem{thm}{Theorem}[section]
\newtheorem{lemma}[thm]{Lemma}

\newtheorem{prop}[thm]{Proposition}
\newtheorem{conj}[thm]{Conjecture}
\newtheorem{innercustomthm}{Theorem}
\newtheorem*{kernelrestatement}{Lemma \ref{lem:zero-eigenvalue}}
\newtheorem*{hksmixingrestatement}{Theorem \ref{thm:mixing-hks}}
\newtheorem*{exclusionrestatement}{Theorem \ref{thm:exclusion-mixing}}
\newtheorem*{comparisonconjecture}{Conjecture \ref{conj:two-thirds}}

\theoremstyle{definition}
\newtheorem{defi}[thm]{Definition}

\newenvironment{customthm}[1]
  {\renewcommand\theinnercustomthm{#1}\innercustomthm}
  {\endinnercustomthm}

\newtheorem{question}[thm]{Question}

\title{Spectral comparison on Specht modules\\and mixing of the interchange process}
\author{Gil Alon}

\begin{document}
\maketitle
\begin{abstract}
We characterize the cases in which the minimal eigenvalues of interchange-process Laplacians on two Specht modules can be compared up to a positive multiplicative constant, independent of the edge rates. In the comparable cases, we establish
an inverse-polynomial lower bound for this constant. We conjecture that the
constant can always be taken to be $2/3$.

We apply these ideas to the mixing times of the interchange and exclusion
processes. Using the comparison method of Hermon, Kozma and Salez, we prove
that our conjecture implies a conjecture of Oliveira on the mixing time of
the interchange process. We also obtain unconditional mixing-time bounds
for both processes.
\end{abstract}

\section{Introduction}
\subsection{Spectral comparison for weighted graphs}
Let $G=([n],w_{ij})$ be a weighted graph with nonnegative weights. The Aldous conjecture, proved by Caputo, Liggett and Richthammer \cite{CLR}, states the following, for any nontrivial partition $\tau \vdash n$:

\begin{equation}\label{eq:aldous}
    \lambda_{\min}(G,\tau) \geq \lambda_{\min} (G,[n-1,1])
\end{equation}

Here $\lambda_{\min}(G,\tau)$ means the minimum eigenvalue of the Laplacian matrix 
\[
L(G,\tau)=\rho_\tau(\sum_{i<j} w_{ij}(id-(ij)))
\]

where $\rho_{\tau}:\mathfrak S_n \rightarrow GL_{d_\tau}(\R)$ is the irreducible representation, or the Specht module, corresponding to $\tau$.

In this setting we can see equation \eqref{eq:aldous} as establishing a relation of spectral dominance between two partitions, or their corresponding irreps. It is only natural to ask whether similar relations hold between other pairs of partitions as well.
This question was considered in the paper \cite{AK} where this spectral dominance order was called the Aldous order. Indeed, it was shown in that paper that such relations hold for additional families of pairs of partitions $\tau, \sigma$. However, even the larger of these families was restricted in the sense that one of the partitions had to have a very large first row (larger than $n-O(\sqrt n)$) and the other had to have a similarly large first column. 

In this paper we ask the question: What if we relax the spectral dominance condition by a multiplicative constant? More concretely, we ask:

\begin{question}
    \label{question1}
    For which pairs of partitions $\tau, \sigma$, does there exist a constant $C>0$ such that the inequality $\lambda_{\min}(G,\tau)\geq C \lambda_{\min}(G,\sigma)$ holds for \emph {any weighted graph} $G$?
\end{question}

\begin{question}
    \label{question2}
     If such a constant $C$ exists for $\tau, \sigma$, what is the \emph{optimal}, i.e., largest, such constant? 
\end{question}

The following theorem provides an answer to the first question:
\begin{thm}\label{thm:comparability}
    Let $\tau,\sigma$ be partitions of $n$. The following are equivalent:
    \begin{enumerate}
        \item There exists $C>0$ such that $\lambda_{\min}(G,\tau)\geq C \lambda_{\min}(G,\sigma)$ for any weighted graph $G$.
        \item $\sigma \trianglerighteq \tau$, where $\trianglerighteq$ is the standard dominance relation among partitions, i.e. $\sum_{i\leq r}\sigma_i \geq \sum_{i \leq r}\tau_i$ for all $r$.
    \end{enumerate}
\end{thm}

Let us define for any nontrivial partitions $\tau,\sigma\vdash n$,
\[
    C(\tau,\sigma) = \inf_{G:\lambda_{\min}(G,\sigma)>0}\frac{\lambda_{\min}(G,\tau)}{\lambda_{\min}(G,\sigma)}
\]

By Theorem \ref{thm:comparability}, we have $C(\tau,\sigma)>0$ when $[n]\neq\sigma\trianglerighteq\tau$. Question \ref{question2} is thus equivalent to finding the numbers $C(\tau,\sigma)$.

We shall prove the following inverse-polynomial bound for these numbers:
\begin{thm}\label{thm:graph-polynomial-bound}
    Let $\tau,\sigma$ be nontrivial partitions of $n$ such that $\sigma \trianglerighteq \tau$. Then
    \[
        C(\tau,\sigma)\geq \frac{1}{n^3}.
    \]
\end{thm}

Simulations and partial results suggest that the numbers $C(\tau,\sigma)$ are uniformly bounded below, independently of $n$:
\begin{conj}\label{conj:two-thirds}
    For all $[n] \neq \sigma \trianglerighteq \tau$, we have
    \[
        C(\tau,\sigma)\geq \frac{2}{3}.
    \]
\end{conj}

This conjecture means that an inequality similar to \eqref{eq:aldous} holds for any $\sigma \trianglerighteq \tau$ and any weighted graph $G$:
\begin{equation}\label{eq:two-thirds}
    \lambda_{\min}(G,\tau)\geq \frac{2}{3}\lambda_{\min}(G,\sigma).
\end{equation}
Apart from being a natural extension of the Aldous conjecture, this conjecture would have implications for the mixing time of the interchange process, as detailed below. The bound $2/3$, if proved, will be tight: We can construct examples, for unbounded values of $n$, where $\lambda_{\min}(G,\tau)/\lambda_{\min}(G,\sigma)=2/3$ (see Theorem \ref{thm:sharp-star-family} below).

We will show that \eqref{eq:two-thirds} holds in the following cases:
\begin{enumerate}
    \item $G$ is a complete graph (i.e. $w_{ij}=c$ for all $i,j$).
    \item More generally, $G$ is a quasi-complete graph (i.e. $w_{ij}$ only depends on $\max\{i,j\}$; this notion was introduced in \cite{AK}).
    \item Clearly, \eqref{eq:two-thirds} holds for any graph $G$ when the Aldous order $\sigma \trianglerighteq_{\mathrm{Aldous}} \tau$ holds. This adds the cases proved in \cite{AK} and of course the case $\sigma=[n-1,1]$ covered by the Caputo--Liggett--Richthammer theorem.
\end{enumerate}
Our simulations did not find a counterexample to this conjecture. They included all unweighted graphs for $n \leq 7$, and many weighted graphs chosen at random.

\subsection{Mixing of the interchange and exclusion processes}\label{subsec:intro-mixing}
Let $G$ be a connected weighted graph on $[n]$, with $n\ge2$. In the
continuous-time interchange process, each vertex carries a distinct label,
and the labels at $i$ and $j$ are exchanged at rate $w_{ij}$. Its state
space is the symmetric group $\mathfrak S_n$.
Write $t_{\mathrm{mix}}(G)$ for the worst-case total-variation mixing time
at distance $1/4$, and put
\[
 \gamma=\lambda_{\min}(G,[n-1,1]).
\]
By the Caputo--Liggett--Richthammer theorem \cite{CLR}, this is also the
spectral gap of the interchange process. 
The general spectral gap bound (see, e.g., \cite[Theorem 20.6]{LP})
therefore gives
\[
 t_{\mathrm{mix}}(G)
 =O\left(\frac{\log(n!)}{\gamma}\right)
 =O\left(\frac{n\log n}{\gamma}\right).
\]
Oliveira conjectured that the factor $n$ can be removed
(see \cite[Conjecture 1]{HS} and \cite[Section 10]{Oliveira}):
\begin{conj}[Oliveira]\label{conj:oliveira}
There is a universal constant $A$ such that every connected weighted graph
$G$ on $n\ge2$ vertices satisfies
\[
 t_{\mathrm{mix}}(G)\le A\frac{\log n}{\gamma}.
\]
\end{conj}
Our approach to the mixing time is based on the full spectrum of the
Laplacian on the regular representation of $\mathfrak S_n$. We seek lower bounds
on the minimal eigenvalue of the Laplacian on each irrep, with an emphasis
on better bounds for higher dimensional irreps. Our main estimate in this
direction is the following.

\begin{thm}\label{thm:first-row}
There is a universal constant $c>0$ such that the following holds.
Let $G$ be a weighted graph on $[n]$ and put
$\gamma=\lambda_{\min}(G,[n-1,1])$. If $\tau\vdash n$ has first row
$\tau_1=n-k$, where $1\le k\le\lfloor n/2\rfloor$, then
\[
 \lambda_{\min}(G,\tau)\ge ck\gamma.
\]
\end{thm}
The proof of Theorem \ref{thm:first-row} is based on the dynamical
comparison principle of Hermon--Kozma--Salez \cite[Proposition 3]{HKS}.
Following their proof strategy, we compare $G$ with its heat-kernel graph
$G_t$ (defined in subsection \ref{subsec:heat}), which is sufficiently close to the normalized complete
graph at a constant multiple of the relaxation time $1/\gamma$.

The restriction $k\le\lfloor n/2\rfloor$ in Theorem \ref{thm:first-row}
can be removed if the $2/3$ conjecture is assumed. Indeed, let $c$ be the
constant in Theorem \ref{thm:first-row}
and put $m=\lfloor n/2\rfloor$. If $k>n/2$, the balanced
two-row partition $\sigma=[n-m,m]$ dominates $\tau$. Hence
\[
 \lambda_{\min}(G,\tau)
 \ge\frac23\lambda_{\min}(G,\sigma)
 \ge\frac{2c}{3}m\gamma
 \ge\frac{c}{3}k\gamma.
\]
We thus get a linear lower bound for every nontrivial partition.
The resulting spectral estimate yields the following mixing bound.

\begin{thm}\label{thm:mixing}
Assume Conjecture \ref{conj:two-thirds}. There is a universal constant
$C>0$ such that every connected weighted graph $G$ on $n\ge2$ vertices,
with ordinary spectral gap $\gamma$, satisfies
\[
 t_{\mathrm{mix}}(G)\le C\frac{\log n}{\gamma}.
\]
In other words, Conjecture \ref{conj:two-thirds} implies Oliveira's
conjecture (Conjecture \ref{conj:oliveira}).
\end{thm}

We also note the following unconditional bound, which is a direct consequence of the Hermon--Kozma--Salez comparison method:

\begin{thm}\label{thm:mixing-hks}
There is a universal constant $C>0$ such that every connected weighted
graph $G$ on $n\ge2$ vertices, with ordinary spectral gap $\gamma$, satisfies
\[
 t_{\mathrm{mix}}(G)\le C\frac{(\log n)^2}{\gamma}.
\]
\end{thm}

Consider now the $r$-particle exclusion process determined by the weighted graph $G$:
$r$ indistinguishable particles occupy distinct vertices of $G$. A particle at $i$ moves to an empty vertex $j$ at rate $w_{ij}$.
Write $t_{\mathrm{mix}}^{\mathrm{EX}(r)}(G)$ for the worst-case
total-variation mixing time of this process at distance $1/4$. It can be determined from the spectrum of the Laplacian on the Specht modules for partitions with two rows. These are covered by Theorem \ref{thm:first-row}, and we get the following unconditional bound:

\begin{thm}\label{thm:exclusion-mixing}
There is a universal constant $C>0$ such that every connected weighted
graph $G$ on $n\ge2$ vertices, with ordinary spectral gap $\gamma$, satisfies
\[
 t_{\mathrm{mix}}^{\mathrm{EX}(r)}(G)
 \le\frac{C\log n}{\gamma}
 \qquad(1\le r\le n-1).
\]
\end{thm}
This bound improves upon the general bound $O((\log n)^2/\gamma)$
following from Oliveira's estimate in \cite{Oliveira} and the bound
$O(\log n \log\log n/\gamma)$ proved by Hermon and Pymar \cite{HP} for regular graphs.

\subsection{The proof ideas}
Theorem \ref{thm:comparability} has two directions. The first implication follows from the following lemma, which may be of independent interest:

\begin{lemma}\label{lem:zero-eigenvalue}
    Let $G=([n],w)$ be a weighted graph, whose connected components have sizes $\tau_1\geq\tau_2\geq\cdots\geq\tau_r>0$. Let $\sigma\vdash n$. Then $\lambda_{\min}(G,\sigma)=0$ if and only if $\sigma\trianglerighteq\tau$.
\end{lemma}

For $\sigma=[n-1,1]$, Lemma \ref{lem:zero-eigenvalue} is equivalent to the fact that $G$ is connected if and only if $\lambda_{\min}(G,\sigma)\neq0$. For a general nontrivial $\sigma$, it describes the condition that $\lambda_{\min}(G,\sigma)\neq0$ in terms of weaker connectivity notions.

Lemma \ref{lem:zero-eigenvalue} is the main resource in proving the direction $(1)\Rightarrow(2)$ in Theorem \ref{thm:comparability}. It implies, in particular, that if $\sigma\not\trianglerighteq\tau$, then for a graph $G$ whose connected component sizes are given by $\tau$, we have $\lambda_{\min}(G,\tau)=0$ while $\lambda_{\min}(G,\sigma)\neq0$.

The second implication in Theorem \ref{thm:comparability} follows from the $1/n^3$ bound in Theorem \ref{thm:graph-polynomial-bound}. The proof of this bound uses the concept of a reducing graph introduced in \cite{AK}. A reducing graph $H$ for the partitions $\sigma,\tau$ had the property that given a weighted graph $G$, if $H\leq G$ and $\lambda_{\min}(G-H,\sigma)\leq\lambda_{\min}(G-H,\tau)$, then $\lambda_{\min}(G,\sigma)\leq\lambda_{\min}(G,\tau)$. Therefore, in proving that $\sigma$ spectrally dominates $\tau$, we could assume that $G$ does not contain copies of $H$.

We develop a parallel concept for inequalities of the form $\lambda_{\min}(G,\sigma)\leq C\lambda_{\min}(G,\tau)$ and use it to prove Theorem \ref{thm:graph-polynomial-bound}.

For quasi-complete graphs, we use the fact that the Laplacian is diagonal
in the Gelfand--Tsetlin basis and there are well known formulas for the corresponding eigenvalues. Given these explicit
eigenvalues, we prove the $2/3$ bound by a combinatorial argument.

For the minimal eigenvalue estimate in Theorem \ref{thm:first-row}, we
start by proving
\[
 \lambda_{\min}(G,\tau)\ge k\gamma-c_k(G),
\]
where $k=n-\tau_1$ and the error term is given by
\[
 c_k(G)=\max_{\substack{B\subseteq[n]\\|B|=k}}
        \sum_{\substack{i,j\in B\\i\ne j}}w_{ij}.
\]
This inequality is proved directly by a Rayleigh quotient comparison.
We then apply the argument of Hermon--Kozma--Salez: replace $G$ by its
heat-kernel graph $G_t$ at time $t=(\log9)/\gamma$. For this graph, the
error term becomes small enough, while the loss in the minimal eigenvalue
is controlled by the dynamical comparison principle.

The remaining mixing time results follow from the eigenvalue bounds,
using standard bounds on the total variation distance.

\paragraph{Acknowledgements.}
The author is grateful to Gady Kozma, Subhajit Ghosh and Doron Puder for helpful discussions. The proofs in this paper were partially assisted by ChatGPT Pro and Astra models, and so were the simulations behind Conjecture \ref{conj:two-thirds}. The author is responsible for the mathematical arguments and their verification.

\section{Preliminaries}
We denote by $\mathfrak S_n$ the symmetric group on $n$ elements. The irreducible representations (or irreps) of this group are parametrized by partitions of $n$, which also correspond to Young diagrams of size $n$. For a partition $\tau \vdash n$, we let $V_\tau$ be the corresponding Specht module, an $\mathbb R[\mathfrak S_n]$ module which realizes the corresponding irrep. It is equipped with an invariant inner product, and we denote by  $\rho_\tau$ the resulting orthogonal representation.
For any orthogonal representation $\rho$ of $\mathfrak S_n$, set
\begin{equation} \label{eq:laplacian_def}
 L(G,\rho)=\sum_{i<j}w_{ij}\bigl(I-\rho((ij))\bigr) 
\end{equation}
and
\[
 \mathcal E_{G,\rho}(f)=\langle f,L(G,\rho)f\rangle.
\]

$L(G,\rho)$ is the \emph{Laplacian operator} 
attached to $G$ at $\rho$, and $ \mathcal E_{G,\rho}(f)$ is the \emph{Dirichlet energy} of $f$ at $G,\rho$. For $f\ne0$, the corresponding Rayleigh quotient is
\[
 R_{G,\rho}(f)=\frac{\mathcal E_{G,\rho}(f)}{\norm f^2}.
\]
The minimum eigenvalue is the minimum of this quotient:
\[
 \lambda_{\min}(G,\rho)=\min_{f\ne0}R_{G,\rho}(f).
\]
We also write $L(G,\tau)$ and $\lambda_{\min}(G,\tau)$, where the
partition $\tau\vdash n$ stands for the corresponding irreducible
representation $\rho_\tau$. We use a representation space in place of
$\rho$ when its action is understood. The analogous notation $\lambda_{\max}$ denotes
the largest eigenvalue. Each $I-\rho((ij))$ is positive semidefinite, so all
these Laplacians are positive semidefinite.

Write $\strep$ for the standard representation of $\mathfrak S_n$
on $\R^n$, in which the symmetric group acts by permuting the
coordinates. It decomposes as the orthogonal sum of the trivial
representation $\mathrm{triv}=\mathbb R\one$ and
$V_{[n-1,1]}=\one^\perp$.
The operator $L(G,\strep)$ is the ordinary graph Laplacian.
For every $g\in\R^n$,
\[
 \mathcal E_{G,\strep}(g)
 =\sum_{i<j}w_{ij}\bigl(g(i)-g(j)\bigr)^2.
\]
We write $\gamma(G)=\lambda_{\min}(G,[n-1,1])$.
It is thus the spectral gap of the ordinary
graph Laplacian.

All inner products on finite function spaces are the standard inner products
unless stated otherwise.
All logarithms are natural.

For self-adjoint operators on the same space, $A \le B$ means that
$\langle f,Af\rangle\le\langle f,Bf\rangle$ for every $f$.
In particular, increasing edge weights increases each representation
Laplacian in this order.

\section{General results for weighted graphs}
We start by proving the characterization of weighted graphs $G$ and partitions $\sigma$ for which $\lambda_{\min}(G,\sigma)=0$.

\begin{kernelrestatement}
    Let $G=([n],w)$ be a weighted graph, whose connected components have sizes $\tau_1\geq\tau_2\geq\cdots\geq\tau_r>0$. Let $\sigma\vdash n$. Then $\lambda_{\min}(G,\sigma)=0$ if and only if $\sigma\trianglerighteq\tau$.
\end{kernelrestatement}

\begin{proof}
Since each operator $I-\rho_\sigma((ij))$ is positive semidefinite, by \eqref{eq:laplacian_def}
\[
    \ker L(G,\sigma)
    =\bigcap_{w_{ij}>0}\ker\bigl(I-\rho_\sigma((ij))\bigr).
\]
Let $B_1,\ldots,B_r$ be the vertex sets of the connected components of the positive-weight support of $G$, including isolated vertices. The transpositions along the edges of a connected graph generate the full symmetric group on its vertices. Therefore the subgroup generated by the positive-weight edges is
\[
    H=\mathfrak S_{B_1}\times\cdots\times\mathfrak S_{B_r},
\]
and $\ker L(G,\sigma)=V_\sigma^H$.

An $H$-equivariant map from the trivial representation $\mathbf1$ to
$V_\sigma$ is determined by the image of $1$, which can be any $H$-fixed
vector. Thus
\[
 V_\sigma^H\cong\operatorname{Hom}_H(\mathbf1,\operatorname{Res}_H^{\mathfrak S_n}V_\sigma).
\]
By Frobenius reciprocity, this space is isomorphic to
\[
 \operatorname{Hom}_{\mathfrak S_n}\bigl(\operatorname{Ind}_H^{\mathfrak S_n}\mathbf1,V_\sigma\bigr)
 =\operatorname{Hom}_{\mathfrak S_n}(M^\tau,V_\sigma),
\]
where $M^\tau=\operatorname{Ind}_H^{\mathfrak S_n}\mathbf1\cong\R[\mathfrak S_n/H]$ is the
Young permutation module. By complete reducibility and Schur's lemma,
the dimension of this space is the multiplicity of $V_\sigma$ in $M^\tau$.

Young's rule identifies this multiplicity with the Kostka number $K_{\sigma,\tau}$, which is positive if and only if $\sigma\trianglerighteq\tau$; see \cite[Section 2.11]{Sagan}. Since $L(G,\sigma)$ is positive semidefinite, its smallest eigenvalue is zero exactly when its kernel is nonzero. This proves the lemma.
\end{proof}

Our next lemma provides the first direction in Theorem \ref{thm:comparability}.

\begin{lemma}\label{lem:dominance-necessary}
Let $\tau,\sigma\vdash n$. If there exists $C>0$ such that
\[
    \lambda_{\min}(G,\tau)\geq C\lambda_{\min}(G,\sigma)
\]
for every weighted graph $G$ on $[n]$, then $\sigma\trianglerighteq\tau$.
\end{lemma}

\begin{proof}
Suppose that $\sigma\not\trianglerighteq\tau$. Let $G$ be a disjoint union of complete graphs whose sizes are the parts of $\tau$. By Lemma \ref{lem:zero-eigenvalue},
\[
    \lambda_{\min}(G,\tau)=0\qquad\text{and}\qquad \lambda_{\min}(G,\sigma)>0.
\]
This contradicts the assumed inequality.
\end{proof}

For the second direction in Theorem \ref{thm:comparability}, we will use an analog of the notion of a reducing graph in \cite{AK}.

\begin{defi}\label{def:alpha-reducing}
Let $\tau,\sigma\vdash n$ and $\alpha>0$. A weighted graph $H$ on $[n]$ is called an \emph{$\alpha$-reducing graph} for $(\tau,\sigma)$ if
\[
    \lambda_{\min}(H,\tau)\geq\alpha\lambda_{\max}(H,\sigma),
\]
where $\lambda_{\max}(H,\sigma)$ denotes the largest eigenvalue of $L(H,\sigma)$.
\end{defi}

\begin{lemma}\label{lem:alpha-reduction}
Let $\alpha>0$ and $\tau,\sigma\vdash n$. Let $H$ be an $\alpha$-reducing graph for $(\tau,\sigma)$ and let $G$ be a weighted graph with $H\leq G$ edgewise. If
\[
    \lambda_{\min}(G-H,\tau)\geq\alpha\lambda_{\min}(G-H,\sigma),
\]
then
\[
    \lambda_{\min}(G,\tau)\geq\alpha\lambda_{\min}(G,\sigma).
\]
Consequently, if $H$ has a positive-weight edge, it suffices to prove this comparison for graphs whose positive-weight support contains no copy of the positive-weight support of $H$.
\end{lemma}

\begin{proof}
Let $v$ be a unit eigenvector of $L(G-H,\sigma)$ with eigenvalue $\lambda_{\min}(G-H,\sigma)$. Since $L(G,\sigma)=L(G-H,\sigma)+L(H,\sigma)$, the Rayleigh quotient gives
\begin{align*}
    \lambda_{\min}(G,\sigma)
    &\leq\langle v,L(G,\sigma)v\rangle\\
    &=\langle v,L(G-H,\sigma)v\rangle+\langle v,L(H,\sigma)v\rangle\\
    &\leq\lambda_{\min}(G-H,\sigma)+\lambda_{\max}(H,\sigma).
\end{align*}
On the other hand, the smallest eigenvalue of a sum of Hermitian matrices is at least the sum of their smallest eigenvalues. Thus
\begin{align*}
    \lambda_{\min}(G,\tau)
    &\geq\lambda_{\min}(G-H,\tau)+\lambda_{\min}(H,\tau)\\
    &\geq\alpha\lambda_{\min}(G-H,\sigma)+\alpha\lambda_{\max}(H,\sigma)\\
    &=\alpha\bigl(\lambda_{\min}(G-H,\sigma)+\lambda_{\max}(H,\sigma)\bigr)\\
    &\geq\alpha\lambda_{\min}(G,\sigma).
\end{align*}
\end{proof}

\begin{lemma}\label{lem:unweighted-graph-gap}
Let $G$ be a connected unweighted graph on $m\geq2$ vertices. Then
\[
    \lambda_{\min}(G,[m-1,1])\geq\frac{2}{m^2}.
\]
\end{lemma}

\begin{proof}
Let $v=(x_1,\ldots,x_m)$ satisfy $\sum_i x_i=0$ and $\sum_i x_i^2=1$, and put
\[
    \mathcal E(v)=\sum_{\{u,w\}\in E(G)}(x_u-x_w)^2.
\]
This is the Rayleigh quotient of the ordinary graph Laplacian at $v$. Its minimum over such vectors is $\lambda_{\min}(G,[m-1,1])$.

For every pair $i<j$, choose a simple path $i=v_0,v_1,\ldots,v_\ell=j$ of length $\ell\leq m-1$. By the Cauchy--Schwarz inequality,
\begin{align*}
    (x_i-x_j)^2
    &=\left(\sum_{h=1}^{\ell}(x_{v_{h-1}}-x_{v_h})\right)^2\\
    &\leq\ell\sum_{h=1}^{\ell}(x_{v_{h-1}}-x_{v_h})^2\\
    &\leq(m-1)\mathcal E(v).
\end{align*}
The last inequality holds because the path uses distinct edges of $G$. Also,
\begin{align*}
    \sum_{i<j}(x_i-x_j)^2
    &=\frac12\sum_{i,j}(x_i-x_j)^2\\
    &=m\sum_i x_i^2-\left(\sum_i x_i\right)^2=m.
\end{align*}
Summing the path estimate over all pairs therefore gives
\[
    m\leq\binom m2(m-1)\mathcal E(v).
\]
Hence $\mathcal E(v)\geq2/(m-1)^2\geq2/m^2$. Taking the minimum over $v$ proves the assertion.
\end{proof}

\begin{lemma}\label{lem:unweighted-specht-gap}
Let $G$ be an unweighted graph on $[n]$. If $\lambda_{\min}(G,\tau)>0$, then
\[
    \lambda_{\min}(G,\tau)\geq\frac{2}{n^2}.
\]
\end{lemma}

\begin{proof}
Let $B_1,\ldots,B_r$ be the connected components of $G$, including isolated vertices, and write $G_i$ for the graph induced on $B_i$. Restrict $V_\tau$ to
\[
    H=\mathfrak S_{B_1}\times\cdots\times\mathfrak S_{B_r}.
\]
This restriction decomposes orthogonally into irreducible tensor products
$W_1\otimes\cdots\otimes W_r$, where $W_i$ is an irreducible representation of $\mathfrak S_{B_i}$. On one such tensor product $W=W_1\otimes\cdots\otimes W_r$, let $L_i$ denote the Laplacian of $G_i$ in the representation $W_i$. Every edge of $G$ lies in a single component, so
\[
    L(G,\tau)|_W
    =\sum_{i=1}^r I_{W_1}\otimes\cdots\otimes I_{W_{i-1}}
      \otimes L_i\otimes I_{W_{i+1}}\otimes\cdots\otimes I_{W_r}.
\]
In particular, if $L_i v_i=\theta_i v_i$ for each $i$, then
\[
    L(G,\tau)(v_1\otimes\cdots\otimes v_r)
    =(\theta_1+\cdots+\theta_r)(v_1\otimes\cdots\otimes v_r).
\]
Taking an eigenbasis in each factor yields an eigenbasis of $W$ consisting of these tensor products. Thus the smallest eigenvalue on $W$ is the sum of the smallest eigenvalues of the $L_i$.

Since $\lambda_{\min}(G,\tau)>0$, there is no $H$-fixed vector. Thus, on each tensor summand, at least one factor $W_i$ is nontrivial. Its component has size $m_i=|B_i|\geq2$. By the Caputo--Liggett--Richthammer theorem \cite{CLR} and Lemma \ref{lem:unweighted-graph-gap}, the smallest eigenvalue contributed by this factor is at least $2/m_i^2\geq2/n^2$. All the other factors contribute nonnegative eigenvalues. Since $W$ was arbitrary, the same lower bound holds on every irreducible $H$-summand of $V_\tau$, and hence for $\lambda_{\min}(G,\tau)$.
\end{proof}

\begin{lemma}\label{lem:forest-reduction}
Let $F$ be an unweighted forest on $[n]$ such that $\lambda_{\min}(F,\tau)>0$. Then $F$ is an $n^{-3}$-reducing graph for $(\tau,\sigma)$ for every partition $\sigma\vdash n$.
\end{lemma}

\begin{proof}
Each $I-\rho_\sigma((ij))$ has operator norm at most $2$, and a forest has at most $n-1$ edges. Hence
\[
    \lambda_{\max}(F,\sigma)\leq2(n-1).
\]
By Lemma \ref{lem:unweighted-specht-gap},
\[
    \lambda_{\min}(F,\tau)\geq\frac{2}{n^2}
    \geq\frac{1}{n^3}\lambda_{\max}(F,\sigma).
\]
\end{proof}

We now recall and prove:

\begin{customthm}{\ref{thm:graph-polynomial-bound}}
Let $\tau,\sigma$ be nontrivial partitions of $n$ such that $\sigma\trianglerighteq\tau$. Then
\[
    C(\tau,\sigma)\geq\frac{1}{n^3}.
\]
\end{customthm}

\begin{proof}
We will prove that
\[
    \lambda_{\min}(G,\tau)\geq\frac{1}{n^3}\lambda_{\min}(G,\sigma)
\]
for every weighted graph $G$, by induction on the number of positive-weight edges. When there are no such edges, both sides are zero.

If $\lambda_{\min}(G,\tau)=0$, let $\pi$ be the partition of connected component sizes of $G$. By Lemma \ref{lem:zero-eigenvalue}, $\tau\trianglerighteq\pi$. Since $\sigma\trianglerighteq\tau$, we also have $\sigma\trianglerighteq\pi$, so $\lambda_{\min}(G,\sigma)=0$ (again by Lemma \ref{lem:zero-eigenvalue}), and the desired inequality holds.

Otherwise, take a spanning tree in each non-singleton connected component of the positive-weight support of $G$, and leave isolated vertices isolated. These trees form an unweighted forest $F$ on $[n]$ with the same connected components as $G$. By Lemma \ref{lem:zero-eigenvalue}, $\lambda_{\min}(F,\tau)>0$. Set
\[
    \varepsilon=\min_{\{i,j\}\in E(F)}w_{ij}>0.
\]
Then $G'=G-\varepsilon F$ has nonnegative weights and strictly fewer positive-weight edges than $G$. By Lemma \ref{lem:forest-reduction}, $F$, and hence $\varepsilon F$, is $n^{-3}$-reducing for $(\tau,\sigma)$. Hence, by Lemma \ref{lem:alpha-reduction}, it is enough to prove the desired comparison for $G'$. This holds by the induction hypothesis.
\end{proof}

We conclude:

\begin{customthm}{\ref{thm:comparability}}
    Let $\tau,\sigma$ be partitions of $n$. The following are equivalent:
    \begin{enumerate}
        \item There exists $C>0$ such that $\lambda_{\min}(G,\tau)\geq C \lambda_{\min}(G,\sigma)$ for any weighted graph $G$.
        \item $\sigma \trianglerighteq \tau$, where $\trianglerighteq$ is the standard dominance relation among partitions, i.e. $\sum_{i\leq r}\sigma_i \geq \sum_{i \leq r}\tau_i$ for all $r$.
    \end{enumerate}
\end{customthm}

\begin{proof}
The implication $(1)\Rightarrow(2)$ is Lemma \ref{lem:dominance-necessary}. Conversely, suppose that $\sigma\trianglerighteq\tau$. If $\sigma=[n]$, then $\lambda_{\min}(G,\sigma)=0$ for every $G$, so the required comparison is immediate. Otherwise both partitions are nontrivial, and Theorem \ref{thm:graph-polynomial-bound} gives the comparison with $C=n^{-3}$.
\end{proof}

\section{The \texorpdfstring{$2/3$}{2/3} conjecture}
By Theorem \ref{thm:graph-polynomial-bound} we have $C(\tau,\sigma)\geq1/n^3$ for all $[n]\neq\sigma\trianglerighteq\tau$. A natural question is, how much can this bound be improved?

By our simulations, and some particular cases detailed below, it seems that there is a universal bound, independent of $n$.

\begin{comparisonconjecture}
For all $[n]\neq\sigma\trianglerighteq\tau$, we have
\[
    C(\tau,\sigma)\geq\frac{2}{3}.
\]
\end{comparisonconjecture}

We will show that Conjecture \ref{conj:two-thirds} holds for quasi-complete graphs. Let us recall their definition from \cite{AK}:

\begin{defi}\label{def:quasi-complete}
A weighted graph $G$ on $[n]$ is \emph{quasi-complete} if there are nonnegative numbers $a_2,\ldots,a_n$ such that
\[
    w_{ij}=a_j\qquad\text{for all }1\leq i<j\leq n.
\]
\end{defi}

Note that in particular, the complete graph (which is also called the mean-field case) is quasi-complete.

\begin{lemma}\label{lem:quasi-complete-spectrum}
Let $G$ be quasi-complete, with weights $a_2,\ldots,a_n$, and let $\nu\vdash n$. For a standard Young tableau $T$ of shape $\nu$, write $c_T(k)$ for the content of the box containing $k$, where the content of a box in row $i$ and column $j$ is $j-i$. Then
\[
    \lambda_{\min}(G,\nu)
    =\min_T\sum_{k=2}^n a_k\bigl(k-1-c_T(k)\bigr),
\]
where the minimum is over all standard Young tableaux of shape $\nu$.
\end{lemma}

\begin{proof}
For $2\leq k\leq n$, let $S_k$ be the unweighted star on $[n]$ with edges $\{i,k\}$ for $i<k$, leaving the other vertices isolated. Then $G=\sum_{k=2}^n a_k S_k$, and
\[
    L(G,\nu)=\sum_{k=2}^n a_k L(S_k,\nu).
\]
By \cite[Lemma 3]{AK}, the Gelfand--Tsetlin basis of $V_\nu$ is a common eigenbasis for all the star Laplacians $L(S_k,\nu)$. Its vectors are indexed by chains of Young diagrams obtained by successively removing one corner. Equivalently, they are indexed by standard Young tableaux: the box removed when the current diagram has $k$ boxes is labelled $k$.

If $v_T$ is the basis vector indexed by $T$, the same lemma gives
\[
    L(S_k,\nu)v_T=\bigl(k-1-c_T(k)\bigr)v_T.
\]
Consequently
\[
    L(G,\nu)v_T
    =\left(\sum_{k=2}^n a_k\bigl(k-1-c_T(k)\bigr)\right)v_T.
\]
As $T$ ranges over all standard tableaux of shape $\nu$, these are all the eigenvalues. Taking the minimum proves the assertion.
\end{proof}

\begin{lemma}\label{lem:corner-removal}
Let $\alpha,\beta\vdash k$, where $k\geq2$, and suppose that $\alpha\trianglerighteq\beta$. For every removable box $b$ of $\beta$, there is a removable box $a$ of $\alpha$ such that
\[
    \alpha\setminus a\trianglerighteq\beta\setminus b
\]
and
\[
    k-1-\operatorname{cont}(a)
    \leq\frac32\bigl(k-1-\operatorname{cont}(b)\bigr).
\]
\end{lemma}

\begin{proof}
For a box $x=(i,j)$ in a diagram of size $k$, write
\[
    d(x)=k-1-\operatorname{cont}(x)=k-j+i-1.
\]
Let $b=(s,t)$, so $t=\beta_s$. We will choose a removable corner $a=(r,u)$ of $\alpha$ with $u\geq t$, while preserving dominance after removal. Moving to the right decreases $d$; only a downward displacement can increase it.

First consider the dominance condition. Removing a box from row $r$ decreases every row-prefix sum starting at $r$ by one. If $r\geq s$, removing $a$ and $b$ therefore preserves dominance automatically. If $r<s$, we need a surplus of at least one in the prefix sums of $\alpha$ at every index from $r$ through $s-1$. Thus any equality of prefix sums above row $s$ forces us to choose $a$ below that equality.

Let $q<s$ be the largest index for which
\[
    \sum_{i=1}^q\alpha_i=\sum_{i=1}^q\beta_i,
\]
allowing $q=0$ with both sums empty. Dominance at the next index gives
\[
    u:=\alpha_{q+1}\geq\beta_{q+1}\geq\beta_s=t.
\]
Starting at row $q+1$, follow the consecutive rows of $\alpha$ of length $u$ down to their last row, say row $r$, and choose $a=(r,u)$. This is a removable corner and $r>q$. If $r<s$, every prefix inequality from $r$ through $s-1$ is strict, by the choice of $q$. Its difference is an integer, hence at least one, as required. If $r\geq s$, removal is already safe. This proves
\[
    \alpha\setminus a\trianglerighteq\beta\setminus b.
\]
The choice of corner is illustrated in Figure \ref{fig:corner-removal}.

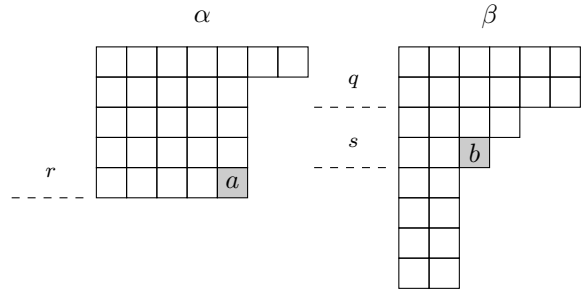
\begin{figure}[H]
\centering
\begin{tikzpicture}[x=4mm,y=-4mm,font=\small]
    \fill[black!20] (4,4) rectangle (5,5);
    \fill[black!20] (12,3) rectangle (13,4);

    \foreach \row/\len in {1/7,2/5,3/5,4/5,5/5}{
        \foreach \col in {1,...,\len}{
            \draw (\col-1,\row-1) rectangle (\col,\row);
        }
    }
    \foreach \row/\len in {1/6,2/6,3/4,4/3,5/2,6/2,7/2,8/2}{
        \foreach \col in {1,...,\len}{
            \draw (10+\col-1,\row-1) rectangle (10+\col,\row);
        }
    }

    \node at (3.5,-1) {$\alpha$};
    \node at (13,-1) {$\beta$};
    \node at (4.5,4.5) {$a$};
    \node at (12.5,3.5) {$b$};

    \draw[dashed] (7.2,2) -- (9.8,2);
    \node[font=\scriptsize] at (8.5,1.15) {$q$};

    \draw[dashed] (7.2,4) -- (9.8,4);
    \node[font=\scriptsize] at (8.5,3.15) {$s$};

    \draw[dashed] (-2.8,5) -- (-0.2,5);
    \node[font=\scriptsize] at (-1.5,4.15) {$r$};
\end{tikzpicture}
\caption{The dashed lines mark the last equality of prefix sums
above $b$ and the rows containing the chosen corners $a$ and $b$.
The chosen corners are shaded.}
\label{fig:corner-removal}
\end{figure}

It remains to compare the costs. Since $u\geq t$,
\[
    d(a)-d(b)=(r-s)-(u-t)\leq r-s.
\]
If $r\leq s$, this gives $d(a)\leq d(b)$. Suppose that $r>s$. Then $t\geq2$: otherwise $b$ would be the last box in the last row of $\beta$, so $\beta$ would have exactly $s$ rows. Dominance would force $\alpha$ to have at most $s$ rows as well, contradicting $r>s$.

Since $a=(r,u)$ and $u\geq t$, the diagram of $\alpha$ contains an $r\times t$ rectangle. Thus $k\geq rt$, and
\begin{align*}
    d(b)&=k-t+s-1\geq rt-t+s-1\\
    &=rt-st+st-t+s-1=t(r-s)+(t+1)(s-1)\\
    &\geq t(r-s).
\end{align*}
The downward displacement therefore increases the cost by at most $d(b)/t$. Since $t\geq2$,
\[
    d(a)\leq d(b)+(r-s)
    \leq\left(1+\frac1t\right)d(b)
    \leq\frac32d(b).
\]
\end{proof}

\begin{thm}\label{thm:quasi-complete-comparison}
Let $G$ be a quasi-complete weighted graph on $[n]$, and let $\tau,\sigma\vdash n$ with $\sigma\trianglerighteq\tau$. Then
\[
    \lambda_{\min}(G,\tau)\geq\frac23\lambda_{\min}(G,\sigma).
\]
\end{thm}

\begin{proof}
Choose a standard tableau $T$ of shape $\tau$ attaining the minimum in Lemma \ref{lem:quasi-complete-spectrum}. Remove its boxes in the order $n,n-1,\ldots,2$. Starting from $\sigma\trianglerighteq\tau$, apply Lemma \ref{lem:corner-removal} at each step to remove a corresponding box from the current diagram on the $\sigma$ side, preserving dominance.

Label each removed box on the $\sigma$ side by the same number as the corresponding box of $T$, and label the last remaining box by $1$. The resulting filling $U$ is a standard tableau of shape $\sigma$. At every step,
\[
    k-1-c_U(k)\leq\frac32\bigl(k-1-c_T(k)\bigr).
\]
Since all $a_k$ are nonnegative, Lemma \ref{lem:quasi-complete-spectrum} gives
\[
    \lambda_{\min}(G,\sigma)
    \leq\sum_{k=2}^n a_k\bigl(k-1-c_U(k)\bigr)
    \leq\frac32\sum_{k=2}^n a_k\bigl(k-1-c_T(k)\bigr)
    =\frac32\lambda_{\min}(G,\tau).
\]
\end{proof}

Let us see that there could not be a better uniform bound than $2/3$.

\begin{thm}\label{thm:sharp-star-family}
For every $m\geq2$, put
\footnote{Note that within the partition notation $a^b$ means $b$ occurrences of $a$.}
\[
    \tau=[2,1^{2m-2}],\qquad \sigma=[2^m],
\]
and let $G$ be the unweighted star on $2m$ vertices. Then $\sigma\trianglerighteq\tau$ and
\[
    \frac{\lambda_{\min}(G,\tau)}{\lambda_{\min}(G,\sigma)}=\frac23.
\]
\end{thm}

\begin{proof}
The dominance relation is immediate. Label the center of the star by $n=2m$. This is a quasi-complete graph with $a_2=\cdots=a_{n-1}=0$ and $a_n=1$. In a standard tableau, the box containing $n$ is removable, and every removable box can contain $n$ in some standard tableau. Lemma \ref{lem:quasi-complete-spectrum} therefore gives
\[
    \lambda_{\min}(G,\nu)=n-1-\max\{\operatorname{cont}(x):x\text{ is a removable box of }\nu\}.
\]
For $\tau=[2,1^{2m-2}]$, the largest removable-box content is $1$, attained at $(1,2)$. For $\sigma=[2^m]$, the only removable box is $(m,2)$, of content $2-m$. Hence
\[
    \lambda_{\min}(G,\tau)=2m-2,\qquad \lambda_{\min}(G,\sigma)=3m-3,
\]
and their ratio is $2/3$.

\end{proof}

\section{Linear lower bounds on minimal eigenvalues}\label{sec:first-row}
We now switch gears and turn to bounds on the minimal eigenvalue
$\lambda_{\min}(G,\tau)$ in terms of
$\gamma:=\gamma(G)=\lambda_{\min}(G,[n-1,1])$. In other words, we seek
lower bounds on the constants $C(\tau,[n-1,1])$ for arbitrary nontrivial
partitions $\tau$. For this purpose, we
realize the relevant Specht modules on ordered tuples and compare their
energies with vertex energies. We then use heat-kernel weights to control
the error in that comparison.

\subsection{Ordered tuples and coordinate sums}
Let $0\le k\le n-1$. Define
\[
 \Om_k=\{(x_1,\ldots,x_k)\in[n]^k:x_a\ne x_b\text{ whenever }a\ne b\}.
\]
The set $\Om_0$ consists of the empty tuple. Put $M_k=\R^{\Om_k}$.
The action of $\mathfrak S_n$ on $\Om_k$ permutes all vertex labels; the corresponding
orthogonal representation on $M_k$ is
\[
 (\rho_k(\pi)f)(x_1,\ldots,x_k)
 =f(\pi^{-1}x_1,\ldots,\pi^{-1}x_k).
\]
For $1\le k\le n-1$ and $1\le r\le k$, define $D_r:M_k\to M_{k-1}$ by
\begin{equation}\label{eq:Dr}
 (D_rf)(b_1,\ldots,b_{k-1})
 =\sum_{i\notin\{b_1,\ldots,b_{k-1}\}}
 f(b_1,\ldots,b_{r-1},i,b_r,\ldots,b_{k-1}).
\end{equation}
Let
\[
 V_k=\bigcap_{r=1}^k\ker D_r.
\]
Changing the summation variable in \eqref{eq:Dr} from $i$ to
$\pi^{-1}i$ gives the intertwining relation
\[
 D_r\rho_k(\pi)=\rho_{k-1}(\pi)D_r.
\]
Thus each $\ker D_r$ is invariant under $\rho_k$, and so is $V_k$.
We regard $V_k$ as a representation of $\mathfrak S_n$ with this restricted action.

\begin{lemma}\label{lem:branch}
Let $1\le k \le n-1$. If $\tau\vdash n$ has $\tau_1=n-k$, then $V_k$ contains a copy of $V_\tau$.
\end{lemma}
\begin{proof}
The action of $\mathfrak S_n$ on $\Om_k$ is transitive. The stabilizer $J_k$ of
$(1,\ldots,k)$ fixes these vertices pointwise and is isomorphic to $\mathfrak S_{n-k}$.
The map
\[
 \pi J_k\longmapsto(\pi(1),\ldots,\pi(k))
\]
is an $\mathfrak S_n$-equivariant bijection. Hence
$M_k\cong\operatorname{Ind}_{J_k}^{\mathfrak S_n}\mathrm{triv}$.
By Frobenius reciprocity, the multiplicity of $V_\nu$ in $M_k$ equals the
multiplicity of the trivial representation in
$\operatorname{Res}_{J_k}^{\mathfrak S_n}V_\nu$.
By repeated branching, this multiplicity counts the sequences of $k$
corner-box removals from $\nu$ that result in the one-row diagram $[n-k]$.
Such a sequence exists if and only if $\nu_1\ge n-k$. Thus
\begin{equation}\label{eq:occurrence}
 V_\nu\text{ occurs in }M_k\quad\Longleftrightarrow\quad\nu_1\ge n-k.
\end{equation}
In particular, $V_\tau$ occurs in $M_k$ and not in $M_{k-1}$.

Choose a copy $U\cong V_\tau$ in $M_k$. The intertwining relation shows
that a nonzero restriction $D_r|_U$ would be injective by irreducibility,
with image isomorphic to $V_\tau$ in $M_{k-1}$, which is impossible.
Therefore $D_r|_U=0$ for every $r$, so $U\subseteq V_k$.
\end{proof}

\subsection{An exact energy identity}
For this section, let $f\in M_k$ be arbitrary. Fix all coordinates of the variable $x\in \Omega_k$ except coordinate $r$ to the corresponding entries of $b\in\Om_{k-1}$ (this is equivalent to requiring that deleting the $r$-th coordinate of $x$ yields $b$). This gives a
function of one vertex, which we extend by zero where it would repeat a
fixed vertex. Explicitly, define $g_{r,b}\in\R^n$ by
\[
 g_{r,b}(i)=
 \begin{cases}
 f(b_1,\ldots,b_{r-1},i,b_r,\ldots,b_{k-1}),
       &i\notin\{b_1,\ldots,b_{k-1}\},\\
 0,&i\in\{b_1,\ldots,b_{k-1}\}.
 \end{cases}
\]
Set
\[
 \SG(f)=\sum_{r=1}^k\sum_{b\in\Om_{k-1}}
 \mathcal E_{G,\strep}(g_{r,b}),
\]
where, for $g\in\R^n$,
\[
 \mathcal E_{G,\strep}(g)
 =\sum_{i<j}w_{ij}\bigl(g(i)-g(j)\bigr)^2.
\]
When $f\in V_k$, the condition $D_rf=0$ makes every $g_{r,b}$ a zero-sum
vector. Therefore, $\SG(f)$ is a sum of energies for
$L(G,[n-1,1])$. Our aim is to bound this sum from above in terms of
$\mathcal E_{G,\rho_k}(f)$. A Rayleigh quotient argument will then compare
$\lambda_{\min}(G,V_k)$ with $\lambda_{\min}(G,[n-1,1])$.

Define
\[
 \CG(f)=\sum_{x\in\Om_k}f(x)^2
 \sum_{r=1}^k\sum_{s\ne r}w_{x_rx_s}.
\]
Also define the contribution from transpositions of two vertices present
in the tuple by
\begin{equation}\label{eq:Eswap}
 \Eswaps(f)=\frac12\sum_{x\in\Om_k}\sum_{r<s}w_{x_rx_s}
 \bigl(f(x)-f((x_r\,x_s)\cdot x)\bigr)^2.
\end{equation}
The inner sum in \eqref{eq:Eswap} is over pairs of coordinates of $x$,
rather than all pairs of vertices in $[n]$.

\begin{lemma}\label{lem:identity}
For every weighted graph $G$ and every $f\in M_k$,
\[
 \mathcal E_{G,\rho_k}(f)=\SG(f)-\CG(f)+\Eswaps(f).
\]
\end{lemma}
\begin{proof}
We expand the left-hand side and transform its terms in three steps.

\noindent\emph{Step 1: Expand the transposition sum.}
By definition, we have
\begin{align*}
 \mathcal E_{G,\rho_k}(f)
 &=\sum_{x\in\Om_k}\sum_{i<j}w_{ij}
 f(x)\bigl(f(x)-f((ij)\cdot x)\bigr)\\
 &=\frac12\sum_{x\in\Om_k}\sum_{i<j}w_{ij}
 \bigl(f(x)-f((ij)\cdot x)\bigr)^2.
\end{align*}
For a fixed $x$, transpositions whose endpoints both lie outside its entries
fix $x$. If exactly one endpoint is present, it has the form $(x_r\,j)$
with $j\notin\{x_1,\ldots,x_k\}$. The remaining transpositions have both
endpoints present and contribute \eqref{eq:Eswap}. Replacing the $r$-th coordinate by $j$, we obtain
\begin{equation}\label{eq:split}
\begin{aligned}
 \mathcal E_{G,\rho_k}(f)
 ={}&\frac12\sum_{x\in\Om_k}\sum_{r=1}^k
 \sum_{j\notin\{x_1,\ldots,x_k\}}w_{x_rj}
 \bigl(f(x)-f(x_1,\ldots,j,\ldots,x_k)\bigr)^2\\
 &+\Eswaps(f).
\end{aligned}
\end{equation}

\noindent\emph{Step 2: Fix all coordinates except one.}
Interchange the sums over $x$ and $r$ in \eqref{eq:split}. For a fixed
$r$, group the tuples $x$ by their $r$-complement $b\in\Om_{k-1}$,
obtained by deleting coordinate $r$, and put $i=x_r$. We can then write
the inner sum symmetrically over $i,j\notin\{b_1,\ldots,b_{k-1}\}$.
The added terms $j=i$ are zero. Hence
\begin{equation}\label{eq:reindex}
\begin{aligned}
 \mathcal E_{G,\rho_k}(f)
 ={}&\frac12\sum_{r=1}^k\sum_{b\in\Om_{k-1}}
 \sum_{i,j\notin\{b_1,\ldots,b_{k-1}\}}w_{ij}
 \bigl(g_{r,b}(i)-g_{r,b}(j)\bigr)^2\\
 &+\Eswaps(f).
\end{aligned}
\end{equation}

\noindent\emph{Step 3: Complete each inner sum to all vertices.}
For fixed $r,b$, the vector $g_{r,b}$ vanishes on the entries of $b$.
The terms with both indices among those entries are therefore zero.
The two orderings of each pair with exactly one index among those entries
have equal contributions, since $w_{ij}=w_{ji}$. Thus
\begin{equation}\label{eq:complete}
\begin{aligned}
 &\frac12\sum_{i,j\notin\{b_1,\ldots,b_{k-1}\}}w_{ij}
 \bigl(g_{r,b}(i)-g_{r,b}(j)\bigr)^2\\
 &\quad=\frac12\sum_{i,j=1}^n w_{ij}
 \bigl(g_{r,b}(i)-g_{r,b}(j)\bigr)^2
 -\sum_{i\notin\{b_1,\ldots,b_{k-1}\}}
   \sum_{j\in\{b_1,\ldots,b_{k-1}\}} w_{ij}g_{r,b}(i)^2\\
 &\quad=\mathcal E_{G,\strep}(g_{r,b})
 -\sum_{i\notin\{b_1,\ldots,b_{k-1}\}}
   \sum_{j\in\{b_1,\ldots,b_{k-1}\}} w_{ij}g_{r,b}(i)^2.
\end{aligned}
\end{equation}
Summing the subtracted terms over all $r$ and $b$, and rearranging by
the participating $k$-tuples $x$, we get
\begin{align}
 &\sum_{r=1}^k\sum_{b\in\Om_{k-1}}
 \sum_{i\notin\{b_1,\ldots,b_{k-1}\}}
 \sum_{j\in\{b_1,\ldots,b_{k-1}\}}w_{ij}g_{r,b}(i)^2\notag\\
 &\qquad=\sum_{x\in\Om_k}f(x)^2\sum_{r=1}^k\sum_{s\ne r}w_{x_rx_s}
 =\CG(f).\label{eq:correction}
\end{align}
Substituting \eqref{eq:complete} and \eqref{eq:correction} into
\eqref{eq:reindex} yields
\begin{align*}
 \mathcal E_{G,\rho_k}(f)
 &=\sum_{r=1}^k\sum_{b\in\Om_{k-1}}
 \mathcal E_{G,\strep}(g_{r,b})-\CG(f)+\Eswaps(f)\\
 &=\SG(f)-\CG(f)+\Eswaps(f). \qedhere
\end{align*}
\end{proof}

\subsection{A spectral lower bound by averaging}
For a $k$-element subset $B\subseteq[n]$, put
\[
 v_G(B)=\sum_{\substack{i,j\in B\\i\ne j}}w_{ij}.
\]
Let $c_k(G)=\max_{|B|=k}v_G(B)$.

\begin{lemma}\label{lem:reduction}
For every weighted graph $G$ and every $1\le k \le n-1$,
\begin{equation}\label{eq:direct}
 \lambda_{\min}(G,V_k)\ge k\gamma(G)-c_k(G).
\end{equation}
\end{lemma}
\begin{proof}
Fix $0\ne f\in V_k$. For each fixed $r$, inserting $i$ in position $r$ is a bijection from the
pairs $(b,i)$ with $i\notin\{b_1,\ldots,b_{k-1}\}$ to $\Om_k$. Hence
\begin{align*}
 \sum_{b\in\Om_{k-1}}\norm{g_{r,b}}^2
 &=\sum_{b\in\Om_{k-1}}
 \sum_{i\notin\{b_1,\ldots,b_{k-1}\}}
 f(b_1,\ldots,b_{r-1},i,b_r,\ldots,b_{k-1})^2\\
 &=\sum_{x\in\Om_k}f(x)^2=\norm f^2.
\end{align*}
Summing over $r$ gives
\begin{equation}\label{eq:totalnorm}
 \sum_{r=1}^k\sum_{b\in\Om_{k-1}}\norm{g_{r,b}}^2=k\norm f^2.
\end{equation}
For the pairs $(r,b)$ with $g_{r,b}\ne0$, define
\[
 \alpha_{r,b}=\frac{\norm{g_{r,b}}^2}{k\norm f^2}.
\]
These weights are positive and sum to $1$ by \eqref{eq:totalnorm}.
Averaging the Rayleigh quotients with these weights gives
\begin{align}
 \sum_{\substack{r,b:\\g_{r,b}\ne0}}
 \alpha_{r,b}R_{G,\strep}(g_{r,b})
 &=\frac{1}{k\norm f^2}
 \sum_{r=1}^k\sum_{b\in\Om_{k-1}}
 \mathcal E_{G,\strep}(g_{r,b})\notag\\
 &=\frac{\SG(f)}{k\norm f^2}.\label{eq:average}
\end{align}
Rearranging the identity of Lemma \ref{lem:identity} and dropping the
nonnegative swap energy, we obtain
\begin{align*}
 \SG(f)
 &=\mathcal E_{G,\rho_k}(f)+\CG(f)-\Eswaps(f)\\
 &\le \mathcal E_{G,\rho_k}(f)+\CG(f).
\end{align*}
The correction satisfies
\begin{equation}\label{eq:Cbound}
 \CG(f)\le c_k(G)\norm f^2.
\end{equation}
Indeed, for each $x\in\Om_k$, its set of entries
$B_x=\{x_1,\ldots,x_k\}$ has size $k$, so
\[
 \sum_{r=1}^k\sum_{s\ne r}w_{x_rx_s}=v_G(B_x)\le c_k(G).
\]
Multiplying each such inequality by $f(x)^2$ and summing gives
\[
 \CG(f)=\sum_{x\in\Om_k}f(x)^2v_G(B_x)
 \le c_k(G)\sum_{x\in\Om_k}f(x)^2=c_k(G)\norm f^2.
\]
Substituting \eqref{eq:Cbound} and collecting terms, we obtain
\begin{align*}
 \SG(f)&\le \mathcal E_{G,\rho_k}(f)+c_k(G)\norm f^2\\
 &=\bigl(R_{G,\rho_k}(f)+c_k(G)\bigr)\norm f^2.
\end{align*}
Consequently, the weighted average in \eqref{eq:average} is at most
$(R_{G,\rho_k}(f)+c_k(G))/k$. One of the Rayleigh quotients is no larger than that
average, so there exist $r,b$ with $g_{r,b}\ne0$ and
\[
 R_{G,\strep}(g_{r,b})
 \le\frac{R_{G,\rho_k}(f)}{k}+\frac{c_k(G)}{k}.
\]
The condition $D_rf=0$ gives
\[
 \sum_{i=1}^n g_{r,b}(i)=(D_rf)(b)=0.
\]
Thus the selected vector is nonzero and perpendicular to the constants.
Its Rayleigh quotient therefore comes from the representation $[n-1,1]$,
and
\[
 \gamma(G)\le R_{G,\strep}(g_{r,b})
 \le\frac{R_{G,\rho_k}(f)}{k}+\frac{c_k(G)}{k}.
\]
Rearranging gives $R_{G,\rho_k}(f)\ge k\gamma(G)-c_k(G)$. Taking the minimum over
${0\ne f\in V_k}$ proves \eqref{eq:direct}.
\end{proof}

\subsection{Bounding the error term}
Our next step is to bound the error term $c_k(G)$ in terms of the spectral
gap $\gamma(G)$. This will eventually
allow us to control it relative to the main term $k\gamma(G)$ in
Lemma \ref{lem:reduction}.

Write $d_G(i)=\sum_{j\ne i}w_{ij}$ for the weighted degree of $i$, and
let $\Delta(G)=\max_i d_G(i)$.

\begin{lemma}\label{lem:small}
For every weighted graph $G$ and every $1\le k \le n-1$,
\begin{equation}\label{eq:ck}
 c_k(G)\le k\Delta(G)-k\gamma(G)\left(1-\frac{k}{n}\right).
\end{equation}
\end{lemma}
\begin{proof}
Fix $B\subseteq[n]$ with $|B|=k$, and decompose its indicator into its
trivial and $[n-1,1]$ components:
\[
 \one_B=\frac{k}{n}\one+h.
\]
Here $h\perp\one$ and, by the Pythagorean theorem,
\[
 \norm h^2=k-\frac{k^2}{n}.
\]
The trivial component has zero energy. Since $h$ belongs to $[n-1,1]$,
we have
\[
 \mathcal E_{G,\strep}(\one_B)
 =\mathcal E_{G,[n-1,1]}(h)
 \ge\gamma(G)\left(k-\frac{k^2}{n}\right).
\]
On the other hand, by definition of the energy,
\[
 \mathcal E_{G,\strep}(\one_B)
 =\sum_{i\in B}\sum_{j\notin B}w_{ij}
 =\sum_{i\in B}d_G(i)-v_G(B).
\]
Rearranging and using the bound on each weighted degree yields
\begin{align*}
 v_G(B)
 &=\sum_{i\in B}d_G(i)-\mathcal E_{G,\strep}(\one_B)\\
 &\le k\Delta(G)-k\gamma(G)\left(1-\frac{k}{n}\right).
\end{align*}
Taking the maximum over $B$ proves \eqref{eq:ck}.
\end{proof}

Combining Lemmas \ref{lem:reduction} and \ref{lem:small}, we obtain
\[
 \lambda_{\min}(G,V_k)
 \ge k\left[\gamma(G)\left(1-\frac{k}{n}\right)
          +\bigl(\gamma(G)-\Delta(G)\bigr)\right].
\]
For the complete graph $G=K_n$ with unit edge weights, we have
$\gamma(K_n)=n$ and $\Delta(K_n)=n-1$, so in this case $\gamma(G)-\Delta(G)$ is negligible relative to $\gamma(G)$ (and, in fact, positive). Hence, for $k\le n/2$, we get the desired bound 
\[
 \lambda_{\min}(K_n,V_k)
 \ge   \frac{k}{2}\gamma(K_n).
\]

For a general weighted graph $G$, following Hermon--Kozma--Salez \cite{HKS}, our strategy is to find
an auxiliary graph $G'$ sufficiently close to a multiple of $K_n$ for these
estimates to be effective, while its Laplacian spectrum on the regular
representation is controlled by that of $G$ through their dynamical
comparison principle. This will allow us to transfer the lower bound back
to $G$ and obtain the estimate in Theorem \ref{thm:first-row}.

\subsection{The Hermon--Kozma--Salez comparison}\label{subsec:heat}
For $t\ge0$, let
\[
 H_t=e^{-tL(G,\strep)}
\]
be the heat kernel of $G$ at time $t$. Let $G_t$, the associated
\emph{heat-kernel graph}, be the weighted graph with edge weights
$H_t(i,j)$ for $i\ne j$.

\begin{lemma}\label{lem:heat}
Assume $\gamma=\gamma(G)>0$. For every $t\ge0$, the matrix $H_t$ is symmetric and stochastic. Its largest
eigenvalue on $\one^\perp$ is $e^{-t\gamma}$, and
\begin{equation}\label{eq:Hgap}
 \gamma(G_t)=1-e^{-t\gamma}.
\end{equation}
Moreover, $\Delta(G_t)\le1$.
\end{lemma}
\begin{proof}
Choose an orthonormal eigenbasis $u_0,\ldots,u_{n-1}$ of $L(G,\strep)$, with
$u_0=n^{-1/2}\one$. Write $L(G,\strep)u_j=\theta_j u_j$.
Then $\theta_0=0$ and $\min_{j\ge1}\theta_j=\gamma$. Moreover,
\[
 H_tu_j=e^{-t\theta_j}u_j.
\]
This identifies the eigenvalues and gives $H_t\one=\one$.
Symmetry follows from symmetry of $L(G,\strep)$.

Choose $a>\Delta(G)$. The entries of $aI-L(G,\strep)$ are
nonnegative, and
\[
 H_t=e^{-at}\sum_{m=0}^{\infty}\frac{t^m}{m!}
 \bigl(aI-L(G,\strep)\bigr)^m.
\]
Thus $H_t$ has nonnegative entries, and by applying it to $\one$ we see that it is stochastic. The weighted degree
of vertex $i$ in $G_t$ is
\[
 d_{G_t}(i)=\sum_{j\ne i}H_t(i,j)=1-H_t(i,i)\le1.
\]
The off-diagonal entries and diagonal entries of the ordinary Laplacian
therefore give
\[
 L(G_t,\strep)=I-H_t.
\]
The eigenvalues already computed now yield \eqref{eq:Hgap}.
\end{proof}

\begin{lemma}[Hermon--Kozma--Salez comparison]\label{lem:comparison}
For every orthogonal representation $\rho$ of $\mathfrak S_n$, and every $t\ge0$,
\begin{equation}\label{eq:HKS}
 L(G_t,\rho)\le t\,L(G,\rho).
\end{equation}
\end{lemma}
\begin{proof}
Hermon--Kozma--Salez \cite[Section 2.3, Proposition 3]{HKS} prove a
comparison of interchange-process Dirichlet forms for symmetric stochastic
matrices. We adapt it here to arbitrary edge rates.

For a symmetric stochastic matrix $Q$, write $G_Q$ for the graph with
edge weights $Q_{ij}$ for $i\ne j$. Let $P$ be an irreducible symmetric
stochastic matrix, and put $P_s=e^{-s(I-P)}$. Their comparison, for every real-valued function $F$ on
$\mathfrak S_n$, is
\begin{equation*}
 \mathcal E_{G_{P_s},\R[\mathfrak S_n]}(F)
 \le s\,\mathcal E_{G_P,\R[\mathfrak S_n]}(F).
\end{equation*}
By continuity, this also holds when $P$ is not irreducible.

The above inequality can be written as an operator inequality

\begin{equation}
    \label{eq:HKSoriginal}
     L(G_{P_s},\R[\mathfrak S_n])
 \le s\, L(G_P,\R[\mathfrak S_n]).
\end{equation}

Given our weighted graph $G$, choose $a>\Delta(G)$ and put
\[
 P=I-\frac1aL(G,\strep).
\]
Its off-diagonal entries are $w_{ij}/a$ and its diagonal entries are
$1-d_G(i)/a$. Thus $P$ is symmetric and stochastic. At time $s=at$,
\[
 P_{at}=e^{-at(I-P)}=e^{-tL(G,\strep)}=H_t.
\]
The weights of $G_P$ are those of $G$ divided by $a$, so
\eqref{eq:HKSoriginal} gives
\[
 L(G_t,\R[\mathfrak S_n])
 \le at\, L(G_P,\R[\mathfrak S_n])
 =t\,L(G,\R[\mathfrak S_n]).
\]
Every irreducible representation of $\mathfrak S_n$ occurs in $\R[\mathfrak S_n]$.
Restricting this inequality to an irreducible summand proves
\eqref{eq:HKS} there; taking orthogonal direct sums proves it for every
$\rho$.
\end{proof}

\subsection{Proof of the first-row bound}
We now recall and prove:

\begin{customthm}{\ref{thm:first-row}}
There is a universal constant $c>0$ such that the following holds.
Let $G$ be a weighted graph on $[n]$ and put
$\gamma=\lambda_{\min}(G,[n-1,1])$. If $\tau\vdash n$ has first row
$\tau_1=n-k$, where $1\le k\le\lfloor n/2\rfloor$, then
\[
 \lambda_{\min}(G,\tau)\ge ck\gamma.
\]
\end{customthm}

\begin{proof}
We may take $c=1/(6\log3)$.
If $\gamma=0$, the claim follows from the nonnegativity of $L(G,V_\tau)$.
Assume $\gamma>0$.

By Lemma \ref{lem:branch}, we may choose an invariant copy $U\cong V_\tau$ inside
$V_k$. Let $f\in U$ be a unit eigenvector of $L(G,\rho_k)|_U$ for its least
eigenvalue $\theta=\lambda_{\min}(G,\tau)$.
Set $t=(\log9)/\gamma$, so that $e^{-t\gamma}=1/9$.
By Lemma \ref{lem:heat}, we have $\gamma(G_t)=1-e^{-t\gamma} = 8/9$ and $\Delta(G_t)\le1$.
Since $k\le n/2$, Lemma \ref{lem:small} now gives
\[
 \frac{c_k(G_t)}{k}
 \le\Delta(G_t)-\gamma(G_t)\left(1-\frac{k}{n}\right)
 \le1-\frac89\cdot\frac12=\frac59.
\]
Applying Lemma \ref{lem:reduction} to $G_t$ yields
\[
 \lambda_{\min}(G_t,V_k)
 \ge k\gamma(G_t)-c_k(G_t)
 \ge\frac{8k}{9}-\frac{5k}{9}=\frac{k}{3}.
\]
Since $f\in V_k$, HKS comparison gives
\begin{align*}
 \lambda_{\min}(G,\tau)
 &=R_{G,\rho_k}(f)\ge\frac1t R_{G_t,\rho_k}(f)\\
 &\ge\frac1t\lambda_{\min}(G_t,V_k)
 \ge\frac{k}{3t}=\frac{k\gamma}{6\log3}.\qedhere
\end{align*}
\end{proof}

Combining Theorem \ref{thm:first-row} and Conjecture \ref{conj:two-thirds}, we get a general lower bound:

\begin{prop}\label{prop:all-rows-conditional}
Assume Conjecture \ref{conj:two-thirds}. There is a universal constant
$c>0$ such that, for every weighted graph $G$ and every nontrivial
partition $\tau\vdash n$,
\[
 \lambda_{\min}(G,\tau)\ge c(n-\tau_1)\gamma.
\]
\end{prop}
\begin{proof}
We may take $c=1/(18\log3)$.
Put $k=n-\tau_1$. If $1\le k\le\lfloor n/2\rfloor$,
the proof of Theorem \ref{thm:first-row} gives the stronger constant $1/(6\log3)$.
Suppose now that $k>n/2$, and put
\[
 m=\lfloor n/2\rfloor,\qquad \sigma=[n-m,m].
\]
This partition dominates $\tau$: its first row is at least $\tau_1$, and
the sum of its first two rows is already $n$. The conjecture and
the explicit bound proved in Theorem \ref{thm:first-row}, applied to $\sigma$, give
\[
 \lambda_{\min}(G,\tau)
 \ge\frac23\lambda_{\min}(G,\sigma)
 \ge\frac23\frac{m\gamma}{6\log3}.
\]
Since $k\le n-1\le2m$, this is at least $k\gamma/(18\log3)$.
\end{proof}

\subsection{An unconditional bound for every first-row length}
We now develop a weaker version of the inequality in Proposition \ref{prop:all-rows-conditional} which does not assume  Conjecture \ref{conj:two-thirds}. The price to pay is a loss of a $\log n$ factor in the bound. This lower bound is a direct consequence of the Hermon--Kozma--Salez approach. It will enable us to prove an unconditional bound for the mixing time.

Let $K_n$ denote the unweighted complete graph. Its representation
Laplacians are scalar, and their eigenvalues admit a useful elementary
lower bound.

\begin{lemma}\label{lem:complete-first-row}
If $\tau\vdash n$ and $k=n-\tau_1$, then $L(K_n,\tau)=\theta_\tau I$,
where $\theta_\tau\ge nk/2$.
\end{lemma}
\begin{proof}
The sum of all transpositions is central, so $L(K_n,\tau)=\theta_\tau I$.
By the Frobenius character-ratio formula \cite[Lemma 7]{DS},
\begin{align*}
 \theta_\tau
 &=\frac12\left(n^2-\sum_i\tau_i^2\right)
      +\sum_i(i-1)\tau_i\\
 &\ge\frac12\left(n^2-n\tau_1\right)=\frac{nk}{2}.
\end{align*}
Here we used $\sum_i\tau_i^2\le\tau_1\sum_i\tau_i=n\tau_1$ and
$\sum_i(i-1)\tau_i\ge0$.
\end{proof}

\begin{prop}\label{prop:all-rows-unconditional}
There is a universal constant $c>0$ such that, for every connected
weighted graph $G$ on $n\ge2$ vertices and every partition $\tau\vdash n$,
\begin{equation}\label{eq:all-rows-unconditional}
 \lambda_{\min}(G,\tau)
 \ge c\frac{(n-\tau_1)\gamma}{\log n}.
\end{equation}
\end{prop}
\begin{proof}
We may take $c=1/8$. More precisely, we prove
\begin{equation}\label{eq:all-rows-unconditional-explicit}
 \lambda_{\min}(G,\tau)
 \ge\frac{(n-\tau_1)\gamma}{4\log(2n)}.
\end{equation}
This implies \eqref{eq:all-rows-unconditional}, since
$\log(2n)\le2\log n$ for $n\ge2$.
Use the single-particle heat kernel $H_s$ and its associated graph $G_s$
from Lemma \ref{lem:heat}, at time
\[
 s=\frac{\log(2n)}{\gamma}.
\]
Let $J$ be the matrix all of whose entries are $1$. Since $J/n$ is
the orthogonal projection onto the constants, $H_s-J/n$ vanishes on
the constants and has eigenvalues $e^{-s\theta}$ on $\one^\perp$,
where $\theta\ge\gamma$ ranges over the nonzero eigenvalues of
$L(G,\strep)$. Thus, for every $v\in\R^n$,
\[
 \left\|\left(H_s-\frac{J}{n}\right)v\right\|
 \le e^{-s\gamma}\norm v=\frac{\norm v}{2n}.
\]
Applying this to the coordinate vector $e_j$ and taking its $i$th
coordinate gives
\[
 \left|H_s(i,j)-\frac1n\right|\le\frac{1}{2n},
\]
so $H_s(i,j)\ge1/(2n)$ for every $i,j$.
The HKS comparison in Lemma \ref{lem:comparison} and edgewise
monotonicity therefore imply, on each Specht module,
\[
 sL(G,\tau)\ge L(G_s,\tau)
 \ge\frac{1}{2n}L(K_n,\tau).
\]
Writing $k=n-\tau_1$ and using Lemma \ref{lem:complete-first-row}, we get
\[
 s\lambda_{\min}(G,\tau)\ge\frac{\theta_\tau}{2n}
 \ge\frac{k}{4},
\]
which is \eqref{eq:all-rows-unconditional-explicit}.
\end{proof}

\section{Mixing of the interchange and exclusion processes}\label{sec:mixing}
We now apply the spectral bounds in the previous section to the analysis
of mixing times.

\subsection{A spectral bound for transitive group actions}\label{subsec:group-action-mixing}
We start by stating a general lemma on the connection between the mixing properties of a group action and eigenvalue bounds on the relevant part of the Laplacian spectrum. It can be found in a slightly different form in a book by Diaconis \cite[Chapter 3F, Lemmas 4 and 5]{Diaconis}, but for completeness we give a self-contained proof based on the argument of \cite[Corollary 5]{HKS}.

\begin{lemma}\label{lem:group-action-mixing}
Let a finite group $\Gamma$ act transitively on a finite set $X$, and let
$u$ be the uniform probability measure on $X$. For rates
$c(g)=c(g^{-1})\ge0$, consider the continuous-time Markov chain with generator
\[
 (\mathcal L f)(x)
 =\sum_{g\in\Gamma}c(g)\bigl(f(g^{-1}x)-f(x)\bigr),
\]
and transition probability matrix for time intervals of length $t$, $P_t=e^{t\mathcal L}$. Write the permutation representation as
\[
 \C^X\cong\C\one\oplus
 \bigoplus_{\rho\ne\mathbf1}V_\rho^{\oplus m_\rho},
\]
where $\rho$ runs over the nontrivial irreducible unitary representations
that occur, and $m_\rho$ is their multiplicity.
Define $L_\rho=\sum_{g\in\Gamma}c(g)(I-\rho(g))$ 
and $\lambda_\rho=\lambda_{\min}(L_\rho)$.
Then, for every $x\in X$ and $t\ge0$,
\begin{equation}\label{eq:group-action-mixing}
 4\|P_t(x,\cdot)-u\|_{\mathrm{TV}}^2
 \le\sum_{\rho\ne\mathbf1}m_\rho d_\rho e^{-2t\lambda_\rho}
\end{equation}
where  $d_\rho=\dim_{\C} V_\rho$.
\end{lemma}

\begin{proof}
Equip $\C^X$ with the inner product of $L^2(X,u)$:
\[
 \langle f,h\rangle
 =\frac1{|X|}\sum_{x\in X}f(x)\overline{h(x)}.
\]
The permutation action is unitary, and the symmetry
$c(g)=c(g^{-1})$ makes $L=-\mathcal L$ self-adjoint.
Moreover, $L\one=0$, so $P_t=e^{-tL}$ fixes the constants.

For each $\rho$, let $W_\rho$ be the sum of its $m_\rho$
copies, of dimension $D_\rho=m_\rho d_\rho$.
Since the permutation representation is unitary, its isotypic
components $W_\rho$ are mutually orthogonal.
Choose an orthonormal eigenbasis
$\phi_{\rho,1},\ldots,\phi_{\rho,D_\rho}$ of $W_\rho$,
with eigenvalues $\lambda_{\rho,\ell}\ge\lambda_\rho$.
Transitivity implies that the only invariant functions
are the constants, so these eigenvectors, together with
$\one$, form an orthonormal basis of $L^2(X,u)$.

Since $P_t=e^{-tL}$, we have
\[
 P_t\phi_{\rho,\ell}
 =e^{-t\lambda_{\rho,\ell}}\phi_{\rho,\ell}.
\]
Expanding an arbitrary function $f$ in the above basis
therefore gives
\[
 (P_tf)(x)
 =\langle f,\one\rangle
 +\sum_{\rho\ne\mathbf1}\sum_{\ell=1}^{D_\rho}
   e^{-t\lambda_{\rho,\ell}}
   \langle f,\phi_{\rho,\ell}\rangle\phi_{\rho,\ell}(x).
\]
By the definition of the inner product,
\[
 \langle f,\phi_{\rho,\ell}\rangle
 =\sum_{y\in X}u(y)f(y)\overline{\phi_{\rho,\ell}(y)}.
\]
Similarly,
\[
 \langle f,\one\rangle=\sum_{y\in X}u(y)f(y).
\]
Substituting and comparing coefficients of $f(y)$ with
\[
 (P_tf)(x)=\sum_{y\in X}P_t(x,y)f(y),
\]
we obtain
\[
 P_t(x,y)
 =u(y)\left(
 1+\sum_{\rho\ne\mathbf1}\sum_{\ell=1}^{D_\rho}
 e^{-t\lambda_{\rho,\ell}}
 \phi_{\rho,\ell}(x)\overline{\phi_{\rho,\ell}(y)}
 \right).
\]
 Therefore,
\[
 \frac{P_t(x,y)}{u(y)}-1
 =\sum_{\rho\ne\mathbf1}\sum_{\ell=1}^{D_\rho}
   e^{-t\lambda_{\rho,\ell}}
   \phi_{\rho,\ell}(x)\overline{\phi_{\rho,\ell}(y)}.
\]
For a fixed $x$ we get, using the orthonormality:
\begin{align*}
 \sum_{y\in X}\frac{|P_t(x,y)-u(y)|^2}{u(y)}
 &=\sum_{y\in X}u(y)
   \left|\frac{P_t(x,y)}{u(y)}-1\right|^2\\
 &=\sum_{\rho\ne\mathbf1}\sum_{\ell=1}^{D_\rho}
   e^{-2t\lambda_{\rho,\ell}}|\phi_{\rho,\ell}(x)|^2\\
 &\le\sum_{\rho\ne\mathbf1}e^{-2t\lambda_\rho}
   \sum_{\ell=1}^{D_\rho}|\phi_{\rho,\ell}(x)|^2.
\end{align*}

For a fixed $\rho$, the sum
\[
 S_\rho(x):=\sum_{\ell=1}^{D_\rho}|\phi_{\rho,\ell}(x)|^2
\]
is independent of the choice of orthonormal basis of $W_\rho$.
The action of any $g\in\Gamma$ takes such a basis to another
orthonormal basis of $W_\rho$, so
$S_\rho(g^{-1}x)=S_\rho(x)$.
By transitivity, $S_\rho$ is constant on $X$.
Averaging over $X$ identifies its value:
\[
 S_\rho(x)
 =\sum_{y\in X}u(y)S_\rho(y)
 =\sum_{\ell=1}^{D_\rho}
   \sum_{y\in X}u(y)|\phi_{\rho,\ell}(y)|^2
=\sum_{\ell=1}^{D_\rho}\|\phi_{\rho,\ell}\|^2
 =D_\rho=m_\rho d_\rho.
\]
Consequently,
\[
 \sum_{y\in X}\frac{|P_t(x,y)-u(y)|^2}{u(y)}
 \le\sum_{\rho\ne\mathbf1}m_\rho d_\rho
      e^{-2t\lambda_\rho}.
\]

Finally, by the Cauchy--Schwarz inequality,
\begin{align*}
 4\|P_t(x,\cdot)-u\|_{\mathrm{TV}}^2
 &=\left(\sum_{y\in X}|P_t(x,y)-u(y)|\right)^2 \\
 &\le |X| \sum_{y\in X}|P_t(x,y)-u(y)|^2
 =\sum_{y\in X}\frac{|P_t(x,y)-u(y)|^2}{u(y)}.
\end{align*}

Combining the last two estimates proves the lemma.\enlargethispage{\baselineskip}
\end{proof}

\subsection{Mixing of the interchange process}
Let $(X_t)_{t\ge0}$ be the interchange process on a connected weighted
graph $G$. The state space of this process is $\mathfrak S_n$, and its generator is
\[
 (\mathcal L_G f)(\pi)
 =\sum_{i<j}w_{ij}\bigl(f((ij)\pi)-f(\pi)\bigr).
\]
Thus $\mathcal L_G=-L(G,\mathrm{reg})$, where $\mathrm{reg}$ denotes the
left regular representation. Let $P_t=e^{t \mathcal L_G}$ be the transition probability matrix for a time interval of length $t$,
and let $u(\eta)=1/n!$ be the uniform probability measure. Define
\begin{equation}\label{eq:d_G(t)}
 d_G(t)=\max_{\pi\in \mathfrak S_n}\frac12\sum_{\eta\in \mathfrak S_n}
       |P_t(\pi,\eta)-u(\eta)|
\end{equation}

and
\[
 t_{\mathrm{mix}}(G)=\inf\{t\ge0:d_G(t)\le1/4\}.
\]
Right translations commute with the generator, so the maximum in \eqref{eq:d_G(t)} can be replaced by taking the value at any $\pi$.

The regular representation decomposes as
\[
 \R[\mathfrak S_n]\cong
 \bigoplus_{\tau\vdash n}V_\tau^{\oplus d_\tau}.
\]
Apply Lemma \ref{lem:group-action-mixing} to the left action of
$\mathfrak S_n$ on itself, using the complexification of this representation,
with rate $w_{ij}$ for the transposition
$(ij)$ and zero rates for all other permutations. Since
$m_\tau=d_\tau$, we obtain, for every $t\ge0$,
\begin{equation}\label{eq:regular-trace}
 4d_G(t)^2
 \le\sum_{\substack{\tau\vdash n\\\tau\ne[n]}}
      d_\tau^2 e^{-2t\lambda_{\min}(G,\tau)}.
\end{equation}

\begin{lemma}\label{lem:dimension-layers}
For every $1\le k\le n-1$,
\[
 \sum_{\substack{\tau\vdash n\\\tau_1=n-k}}d_\tau^2\le n^{2k}.
\]
\end{lemma}
\begin{proof}
The ordered-tuple module $M_k=\R^{\Om_k}$ has dimension
$$n(n-1)\cdots(n-k+1) \le n^k.$$ By \eqref{eq:occurrence}, it contains at least
one copy of every $V_\tau$ with $\tau_1=n-k$. Hence
\[
 \sum_{\tau_1=n-k}d_\tau\le\dim M_k\le n^k.
\]
The sum of the squares of these nonnegative dimensions is no greater
than the square of their sum.
\end{proof}

\begin{prop}\label{prop:linear-to-mixing}
There is a universal constant $C$ such that if 
 $a>0$ and
\begin{equation}\label{eq:all-layer-hypothesis}
 \lambda_{\min}(G,\tau)\ge a(n-\tau_1)\gamma
 \qquad\text{for every }\tau\ne[n]
\end{equation}
then 
\begin{equation}\label{eq:general-mixing}
 t_{\mathrm{mix}}(G)\le\frac{C\log n}{a\gamma}.
\end{equation}
\end{prop}
\begin{proof}
We may take $C=3$. In fact, we prove the sharper estimate
\begin{equation}\label{eq:general-mixing-explicit}
 t_{\mathrm{mix}}(G)
 \le\frac{1}{a\gamma}\left(\log n+\frac12\log5\right).
\end{equation}
Grouping the right-hand side of \eqref{eq:regular-trace} by $k=n-\tau_1$ and using
Lemma \ref{lem:dimension-layers}, we obtain
\[
 4d_G(t)^2\le\sum_{k=1}^{n-1}n^{2k}e^{-2ak\gamma t}
 =\frac{q-q^n}{1-q}
\]
where $q=n^2e^{-2a\gamma t}$. For the value of $t$ on the right-hand side of \eqref{eq:general-mixing-explicit}, we have $q = n^2 e^{-(2 \log n + \log 5) }= \frac 1 5$, so
\[
 4d_G(t)^2 \le \frac{q}{1-q}=\frac 1 4
\]
Thus $d_G(t)\le1/4$, proving \eqref{eq:general-mixing-explicit}.
Since $n\ge2$,
\[
 \log n+\frac12\log5\le3\log n.
\]
Hence \eqref{eq:general-mixing} holds with $C=3$.
\end{proof}

Using Proposition \ref{prop:linear-to-mixing} and our spectral bounds, we can now prove the conditional and unconditional mixing time bounds.

\begin{customthm}{\ref{thm:mixing}}
Assume Conjecture \ref{conj:two-thirds}. There is a universal constant
$C>0$ such that every connected weighted graph $G$ on $n\ge2$ vertices,
with ordinary spectral gap $\gamma$, satisfies
\[
 t_{\mathrm{mix}}(G)\le C\frac{\log n}{\gamma}.
\]
In other words, Conjecture \ref{conj:two-thirds} implies Oliveira's
conjecture (Conjecture \ref{conj:oliveira}).
\end{customthm}

\begin{proof}
We may take $C=54\log3$. More precisely, we prove
\[
 t_{\mathrm{mix}}(G)\le
 \frac{18\log3}{\gamma}
 \left(\log n+\frac12\log5\right).
\]
The proof of Proposition \ref{prop:all-rows-conditional} gives
\eqref{eq:all-layer-hypothesis} with $a=1/(18\log3)$.
Apply \eqref{eq:general-mixing-explicit} from the proof of
Proposition \ref{prop:linear-to-mixing}.
Since $\log n+(1/2)\log5\le3\log n$ for $n\ge2$, the stated bound
follows with $C=54\log3$.
\end{proof}

\begin{hksmixingrestatement}
There is a universal constant $C>0$ such that every connected weighted
graph $G$ on $n\ge2$ vertices, with ordinary spectral gap $\gamma$, satisfies
\[
 t_{\mathrm{mix}}(G)\le C\frac{(\log n)^2}{\gamma}.
\]
\end{hksmixingrestatement}

\begin{proof}
Let $c>0$ be the universal constant in
Proposition \ref{prop:all-rows-unconditional}, and put $a=c/\log n$.
Here $a$ depends on $n$, but not on the partition $\tau$.
For every nontrivial partition $\tau\vdash n$, Proposition \ref{prop:all-rows-unconditional} gives
\[
 \lambda_{\min}(G,\tau)
 \ge\frac{c(n-\tau_1)\gamma}{\log n}
 =a(n-\tau_1)\gamma.
\]
Thus the hypothesis of Proposition \ref{prop:linear-to-mixing} holds
with this choice of $a$. Applying that proposition, we obtain
\[
 t_{\mathrm{mix}}(G)
 \le\frac{C_0\log n}{a\gamma}
 =\frac{C_0}{c}\frac{(\log n)^2}{\gamma},
\]
where $C_0$ is universal. The theorem follows with $C=C_0/c$,
which is also universal.
\end{proof}

\subsection{An unconditional logarithmic bound for the exclusion process}\label{subsec:exclusion}
Let $\mathcal X_r=\{A\subseteq[n]:|A|=r\}$, $W_r=\R^{\mathcal X_r}$, and let $\rho_r^{\mathrm{EX}}$ be the natural permutation action of $\mathfrak S_n$ on $W_r$. The generator of the $r$-particle exclusion process is
\[
 (\mathcal L_G^{\mathrm{EX}(r)}f)(A)
 =\sum_{i\in A}\sum_{j\notin A}w_{ij}
   \bigl(f((A\setminus\{i\})\cup\{j\})-f(A)\bigr)
 =-L(G,\rho_r^{\mathrm{EX}})f(A).
\]
The uniform measure $u_r$ on $\mathcal X_r$ is stationary.
The map $A\mapsto A^c$ intertwines the processes with $r$ and $n-r$
particles, so their mixing times are equal. It therefore suffices to
consider $1\le r\le\lfloor n/2\rfloor$.

\begin{lemma}\label{lem:exclusion-decomposition}
For $0\le r\le\lfloor n/2\rfloor$, the orthogonal decomposition of $W_r$
into Specht modules is
\begin{equation}\label{eq:exclusion-decomposition}
 W_r\cong\bigoplus_{j=0}^r V_{[n-j,j]},
\end{equation}
where the summand for $j=0$ is the trivial representation $V_{[n]}$.
\end{lemma}
\begin{proof}
See \cite[Chapter 7B, pp. 138--139]{Diaconis}.
\end{proof}

We now recall and prove:

\begin{exclusionrestatement}
There is a universal constant $C>0$ such that every connected weighted
graph $G$ on $n\ge2$ vertices, with ordinary spectral gap $\gamma$, satisfies
\[
 t_{\mathrm{mix}}^{\mathrm{EX}(r)}(G)
 \le\frac{C\log n}{\gamma}
 \qquad(1\le r\le n-1).
\]
\end{exclusionrestatement}

\begin{proof}
Put $a=1/(6\log3)$. We may take $C=12\log3$; more precisely, we prove that
\begin{equation}\label{eq:exclusion-mixing-explicit}
 t_{\mathrm{mix}}^{\mathrm{EX}(r)}(G)
 \le\frac{3\log3}{\gamma}\left(\log n+\log5\right).
\end{equation}
By particle-hole symmetry, assume $r\le n/2$. Apply
Lemma \ref{lem:group-action-mixing} to the action of $\mathfrak S_n$ on
$\mathcal X_r$, using the complexification of $W_r$, with rate $w_{ij}$ for the transposition $(ij)$ and
zero rates for all other permutations. This gives exactly the exclusion
generator displayed above. By Lemma \ref{lem:exclusion-decomposition},
the nontrivial summands are $V_{[n-j,j]}$, $1\le j\le r$, each with
multiplicity one and dimension $d_j=\dim V_{[n-j,j]}$. Consequently, if $K_t$ is the
transition probability matrix of the exclusion process, then for every initial set
$A\in\mathcal X_r$,
\begin{equation}\label{eq:exclusion-spectral-sum}
 4\|K_t(A,\cdot)-u_r\|_{\mathrm{TV}}^2
 \le\sum_{j=1}^r d_j e^{-2t\lambda_{\min}(G,[n-j,j])}.
\end{equation}

Taking dimensions in \eqref{eq:exclusion-decomposition} with $r=j$ gives
\[
 \binom nj=\sum_{i=0}^j d_i.
\]
Subtracting the corresponding identity for $j-1$, we obtain, for $j\ge1$,
\begin{equation}\label{eq:exclusion-dimensions}
 d_j=\binom nj-\binom n{j-1}\le\binom nj\le n^j.
\end{equation}
For $1\le j\le r$, the explicit estimate proved in
Theorem \ref{thm:first-row} gives
\[
 \lambda_{\min}(G,[n-j,j])\ge aj\gamma.
\]
Using \eqref{eq:exclusion-dimensions} in
\eqref{eq:exclusion-spectral-sum}, we obtain
\[
 4\|K_t(A,\cdot)-u_r\|_{\mathrm{TV}}^2
 \le\sum_{j=1}^r n^j e^{-2aj\gamma t}
 =\frac{q-q^{r+1}}{1-q},
\]
where $q=ne^{-2a\gamma t}$.
For the value of $t$ on the right-hand side of
\eqref{eq:exclusion-mixing-explicit}, we have
$q=ne^{-(\log n+\log5)}=\frac15$, so
\[
 4\|K_t(A,\cdot)-u_r\|_{\mathrm{TV}}^2
 \le\frac{q}{1-q}=\frac14.
\]
Thus the total-variation distance is at most $1/4$ for every
initial set $A$, proving \eqref{eq:exclusion-mixing-explicit}.
Since $n\ge2$,
\[
 \log n+\log5\le4\log n.
\]
Hence the claimed mixing-time bound holds with $C=12\log3$.
\end{proof}

\end{document}